\documentclass{article}
\usepackage[english]{babel}
\usepackage[utf8]{inputenc}
\usepackage[T1]{fontenc}
\usepackage{graphicx}
\usepackage{geometry}
\usepackage{sectsty}
\usepackage{amsthm,amsfonts}
\usepackage{amsmath}
\usepackage{mathabx}
\usepackage{accents}
\usepackage{titlesec}
\usepackage{mathtools}
\mathtoolsset{showonlyrefs=true}

\titleformat*{\section}{\Large\bfseries}
\titleformat*{\subsection}{\large\bfseries}
\titleformat*{\subsubsection}{\large\bfseries}
\titleformat*{\paragraph}{\large\bfseries}
\titleformat*{\subparagraph}{\large\bfseries}

\newtheorem{teo}{Theorem}
\newtheorem{lema}[teo]{Lemma}

\newtheorem{prop}[teo]{Proposition}

\newtheoremstyle{mytheoremstyle} 
{\topsep}                    
{\topsep}                    
{}                   
{}                           
{\scshape}                   
{.}                          
{.5em}                       
{}  

\theoremstyle{mytheoremstyle} 
\theoremstyle{mytheoremstyle} 
\numberwithin{equation}{section}
\newcommand{\real}{\mathbb{R}}

\newcommand{\nat}{\mathbb{N}}

\newcommand \ben {\begin{equation}}
\newcommand \een {\end{equation}}
\newcommand \be {\begin{equation*}}
\newcommand \ee {\end{equation*}}
\newcommand \bi {\begin{itemize}}
\newcommand \ei {\end{itemize}}

\newcommand{\proj}{\mathbb{P}}

\DeclareMathOperator*{\Diver}{div}

\title{\textbf{Existence and stability of spatial plane waves for the incompressible Navier-Stokes in $\real^3$}}
\author{Simão Correia and Mário Figueira}
\begin{document}
\maketitle
\begin{abstract}
We consider the three-dimensional incompressible Navier-Stokes equation on the whole space. We observe that this system admits a $L^\infty$ family of global spatial plane wave solutions, which are connected with the two-dimensional equation. We then proceed to prove local well-posedness over a space which includes $L^3(\real^3)$ and these solutions. Finally, we prove $L^3$-stability of spatial plane waves, with no condition on their size.
\end{abstract}
\section{Introduction}
In this work, we consider the Cauchy problem for the incompressible Navier-Stokes equation on $\real^d$,
\begin{equation}\label{NS}\tag{NS}
\left\{\begin{array}{l}
u_t - \Delta u + (u\cdot\nabla)u = \nabla p, \quad t>0,\ u(t,\cdot):\real^d\to \real^d\\
\Diver u =0\\
u(0)=u_0
\end{array}\right.
\end{equation}
We shall focus on the $d=3$ case. We seek to study spatial plane waves, that is, 
\begin{equation}\label{ansatz}
u(t,x,y,z)=g(t,x-cy,z),\ p(t,x,y,z)=q(t,x-cy,z),\ u_0(x,y,z)=g_0(x-cy,z),\ c\in\real.
\end{equation}
We shall refer to $c$ as the speed of the wave and $g,q$ as the wave profiles. The idea of considering such solutions first appeared in \cite{ms1} in the context of the hyperbolic nonlinear Schrödinger equation. Since the existence of such solutions is quite trivial in such a framework, the attention was then directed to the local well-posedness over a space which includes $H^1$ functions and spatial plane waves. Finally, it was proven, in some cases, that $H^1$ perturbations of spatial plane waves are stable. These ideas were later developed for the nonlinear Schrödinger equation (see \cite{ms2}), where one considers superpositions of waves with different speeds (either a numerable collection or a continuous one).

The generality of such results made us search for other models where one could try to apply these ideas, such as the \eqref{NS}. A considerable change in the framework is observed: on one hand, the change from a dispersive equation to a diffusive one; on the other, the passage from a scalar equation to a system. Moreover, there is an intrinsic interest in obtaining existence and stability results for the three-dimensional (NS) (and so this work is not simply an academic problem).

The existence of a class of global spatial plane waves is proven by observing that the profile satisfies a two-dimensional (NS) system (cf. Proposition 4). Naturally, these solutions will not belong to $L^p(\real^3)$ for any $p<\infty$. However, under some regularity assumptions over the initial data $g_0$, they will belong to $L^\infty(\real^3)$. We then derive a local well.posedness result over a space which includes these global solutions and $L^3(\real^3)$ (cf. Theorem \ref{lwpE}). Finally, we prove the stability of spatial plane waves under $L^3(\real^3)$ perturbations, without any smallness condition of the profile of the wave. As a consequence, our result proves that, if a flow in $\real^3$ is horizontal (that is, it has no vertical components and is independent of the height), then localized perturbations of such a flow give rise to global flows, which remain close to the horizontal one.

This paper is organized as follows: in a first section, we recall some well-known classical results. We obtain a family of global spatial plane waves in section 3. The main results are then stated and proved in section 4. Finally, we make some comments regarding superpositions of spatial plane waves and the extension of these results to a very similar evolution problem, the complex Ginzburg-Landau equation.

\section{Preliminaries}
We shall denote by $\mathbb{P}$ the projection operator over the set of divergence-free functions, which is defined on any $L^p$, $1<p<\infty$. Define $A=-\Delta\mathbb{P}=-\mathbb{P}\Delta$. Recall the basic estimates for the projected heat kernel $e^{tA}$, valid for any $1< q\le p< \infty$:
\begin{equation}\label{estimativalinear1}
\|e^{tA}u\|_p\lesssim t^{-\frac{1}{2}\left(\frac{d}{q}-\frac{d}{p}\right)}\|u\|_q,
\end{equation}
\begin{equation}\label{estimativalinear2}
\|\nabla e^{tA}u\|_p\lesssim t^{-\frac{1}{2}\left(1+\frac{d}{q}-\frac{d}{p}\right)}\|u\|_q.
\end{equation}
Throughout this work, we use Kato's definition for mild solutions (see \cite{kato}, \cite{katofujita}): a solution is any function $u\in C([0,T); \proj\left(L^p(\real^d)\right))$, for some $0<t\le\infty$ and $1< p< \infty$, which satisfies the integral equation
\begin{equation}\label{duhamel}
u(t)=e^{t\Delta}u_0 - \int_0^t e^{(t-s)A}\proj\left((u\cdot\nabla)u\right)(s)ds.
\end{equation}

We now recall some important results in the classical theory for the (NS).

\begin{teo}[see \cite{kato}]\label{kato}
Let $u_0\in \proj L^d(\real^d)$. Then there exists $T>0$ and a unique solution  $u\in C([0,T); \proj\left(L^d(\real^d)\right))$ of \eqref{NS}. It satisfies
\begin{equation}\label{decayu}
t^{(1-d/p)/2}u\in BC([0,T),\proj L^p(\real^d)), d\le p\le \infty,
\end{equation}
\begin{equation}\label{decaygradu}
t^{(1-d/2p)}\nabla u\in BC([0,T),\proj (L^p(\real^d))^d), d\le p< \infty,
\end{equation}
$$
\sup_{t>0, 3\le p\le \infty}\left\{t^{(1-d/p)/2}\|u(t)\|_p\right\}\lesssim 2\|u_0\|_d.
$$
Furthermore, there exists $\epsilon>0$ such that, if $\|u_0\|_{d}\le \epsilon$, the solution $u$ is global (that is, $T=+\infty$).
\end{teo}

\begin{teo}[for e.g., \cite{lemarie}]\label{globexistd2}
Let $u_0\in \proj H^s(\real^2)$, $s\ge 1$. Then there exists a unique solution of (NS) on
$$
C([0,\infty);\proj H^s(\real^2))\cap L^2_{loc}((0,\infty), \proj H^{s+1}(\real^2)).
$$
\end{teo}

\begin{lema}\label{lema:estimativa}
Let $u_0\in\proj H^s(\real^2)$, $s\ge1$. Then the solution $u$ of (NS) with initial data $u_0$ given by Theorem \ref{globexistd2} satisfies
$$
\|u(t)\|_2\to 0,\quad t\to\infty.
$$
Furthermore, for $t_0$ sufficiently large,
$$
\|u(t)\|_\infty\lesssim \frac{2\|u(t_0)\|_2}{(t-t_0)^{1/2}}.
$$
\end{lema}
\begin{proof}
From Wiegner's theorem (see, for example \cite[Theorem 26.1]{lemarie}), we know that $\|u(t)\|_2\to 0$ as $t\to\infty$. Hence, for $t_0$ sufficiently large, $u(t_0)$ satisfies the conditions of Theorem \ref{kato}. 
\end{proof}


\section{Existence of spatial plane waves}

Introducing the spatial plane wave ansatz \eqref{ansatz} into the three dimensional (NS) and setting
\begin{align*}
h_1(t,w,z)&=\frac{1}{\sqrt{1+c^2}}\left(g_1(t,\sqrt{1+c^2}w,z)-cg_2(t,\sqrt{1+c^2}w,z)\right), \\h_2(t,w,z)&=g_3(t,\sqrt{1+c^2}w,z),\quad \rho(t,w,z)=q(t,\sqrt{1+c^2}w,z),
\end{align*}
it is easy to check that $h,=(h_1,h_2)$ is a solution of the two-dimensional (NS) with pressure $\rho$ and initial data
\begin{align*}
(h_0)_1(w,z)&=\frac{1}{\sqrt{1+c^2}}\left((g_0)_1(\sqrt{1+c^2}w,z)-c(g_0)_2(t,\sqrt{1+c^2}w,z)\right),\\ (h_0)_2(w,z)&=(g_0)_3(\sqrt{1+c^2}w,z).
\end{align*}
To determine the wave profile $g$, one must now solve the equation for
$$
\tilde{g}_1(t,w,z)=\sqrt{1+c^2}g_1(t,\sqrt{1+c^2}w,z).
$$
A simple computation gives
\begin{equation}\label{eqg1}
\left\{\begin{array}{l}
\frac{\partial \tilde{g}_1}{\partial t} - \Delta \tilde{g}_1 + h_1(\tilde{g}_1)_w + h_2(\tilde{g}_1)_z = \frac{\partial \rho}{\partial w}\\ \tilde{g}(0,w,z)=\sqrt{1+c^2}(g_0)_1(\sqrt{1+c^2}w,z)
\end{array}\right.
\end{equation}
which is a linear heat equation with a source term. The problem is the insuficient information on the term $\partial q/\partial w$, for it does not provide the estimations needed for the existence proof of $\tilde{g}_1$. However, we observe that \eqref{eqg1} is also the equation satisfied by $h_1$. This means that $\tilde{g}_1=h_1$ satisfies \eqref{eqg1} if the initial data coincide. Even though we lose some freedom in the choice of the initial profile, this degeneracy allows us to obtain a solution of the three-dimensional (NS), given by
\begin{align*}
\phi(t,x,y,z)&=\left(\frac{1}{\sqrt{1+c^2}}h_1\left(t,\frac{x-cy}{\sqrt{1+c^2}},z\right), -\frac{c}{\sqrt{1+c^2}}h_1\left(t,\frac{x-cy}{\sqrt{1+c^2}},z\right), h_2\left(t,\frac{x-cy}{\sqrt{1+c^2}},z\right)\right)\\&=: W[h(t)](x,y,z).
\end{align*}
For any $s>2$, set
$$
X_c^s=\left\{\phi\in (L^1_{loc}(\real^3))^3: \phi=W[h],\ h\in (H^s(\real^2))^2 \right\}.
$$
endowed with the induced norm $\|\phi\|_{X_c^s}:=\|h\|_{H^s}$. An elementary computation shows that if $\phi=W[h]\in X_c^s$,
$$
(\Diver\phi)(x,y,z)=(\Diver h)\left(\frac{x-cy}{\sqrt{1+c^2}}, z\right).
$$
On the other hand, write $h=\proj h + \nabla_{w,z}\Psi$. Then
$$
W[\nabla_{w,z}\Psi] = \nabla_{x,y,z} \tilde{\Psi}, \quad \tilde{\Psi}(x,y,z)=\Psi\left(\frac{x-cy}{\sqrt{1+c^2}},z\right).
$$
This implies that $\phi=W[h]$ can be written (uniquely) as $\phi=W[\proj h] + \nabla \tilde{\Psi}$, that is, as the sum of a divergence-free vector field plus a gradient term. In this way, we define naturally $\proj \phi = W[\proj h]$ and
$$
\proj X_c^s=\left\{\phi\in (L^1_{loc}(\real^3))^3: \phi=W[h],\ h\in (\proj H^s(\real^2))^2 \right\}.
$$
These observations show that, to solve \eqref{NS} on $\proj X_c^s$, it suffices to solve the projected equation, whose weak formulation is given by \eqref{duhamel}. We may finally state the existence of global spatial plane waves for the three-dimensional (NS), using Theorem \ref{globexistd2}:
\begin{prop}\label{prop:existpw}
Fix $s\ge 1$ and $c\in\real$. Given $\phi_0\in \proj X_c^s$, there exists a unique solution $\phi\in C([0,\infty), \proj X_c^s)$ of (NS), which is obtained by solving the two-dimensional (NS) system for its profile.
\end{prop}

\section{Main results}
Throughout this section, $s>2$ and $c\in \real$ will be fixed. Set
$$
E=\proj L^3(\real^3) \oplus \proj X^s_c.
$$

\begin{teo}[Local well-posedness over $E$]\label{lwpE}
Given $u_0=v_0+\phi_0\in E$, there exists $T>0$ and a solution $u\in C([0,T); E)$ of (NS). If $\phi$ is the solution of (NS) with initial data $\phi_0$ (cf. Proposition \ref{prop:existpw}),
\begin{equation}\label{decayv1}
t^{(1-3/p)/2}(u-\phi)\in BC([0,T],\proj L^p(\real^3)), 3\le p\le \infty,
\end{equation}
\begin{equation}\label{decaygradv2}
t^{(1-3/2p)}\nabla (u-\phi)\in BC([0,T],\proj (L^p(\real^3))^3), 3\le p< \infty,
\end{equation}
Moreover, for any $T^*>0$, there exists $\epsilon=\epsilon(T^*,\phi_0)>0$ such that, if $\|v_0\|_3 <\epsilon$, $T=T^*$ and $\|v(T^*)\|_3\le 2\epsilon$.
\end{teo}

\begin{proof} \textit{Step 1. Introduction.}
If one takes the difference between the Duhamel's formulae for $u$ and $\phi$, one arrives at an equivalent problem for $v=u-\phi$:
\begin{equation}\label{duhamelv}
v(t)=e^{t\Delta}v_0 - \int_0^t e^{(t-s)\Delta}\proj\left((v\cdot\nabla)v+(\phi\cdot\nabla)v+(v\cdot\nabla)\phi\right)(s)ds.
\end{equation}
We follow closely the technique in \cite{kato} to solve this equation using the method of successive approximations: given $T^*>0$, consider the sequence
$$
v_{n+1}=v_1 + Gv_n, \ n\ge1,
$$
where
\begin{equation}
v_1=e^{t\Delta}v_0,\quad Gv:=-\int_0^t e^{(t-s)\Delta}\proj\left((v\cdot\nabla)v+(\phi\cdot\nabla)v+(v\cdot\nabla)\phi\right)(s)ds,\ t<T^*.
\end{equation}
The goal will be to obtain the existence of a limit $v$, which will satisfy Duhamel's formula. We do this if $\|v_0\|_3$ is small; if it is not, then one makes $T^*$ small and carries out the same proof.

\textit{Step 2. Estimations on successive approximations.} Fix $L>0$ large enough. Now we prove, by induction, the existence of such $v_n$, satisfying
\begin{equation}\label{estimativainducao}
e^{-Lt}t^{(1-\gamma)/2}v_n\in BC([0,T^*];\proj L^{3/\gamma})\mbox{ with norm smaller than }K_n, \ 0<\gamma\le 1,
\end{equation}
\begin{equation}\label{estimativainducao2}
e^{-Lt}t^{1/2}\nabla v_n\in BC([0,T^*];\proj L^{3})\mbox{ with norm smaller than }K'_n,
\end{equation}
Starting with $n=1$, \eqref{estimativainducao} and \eqref{estimativainducao2} follow directly from \eqref{estimativalinear1} and \eqref{estimativalinear2}, with
$$
K_1,K_1'\lesssim \|v_0\|_3.
$$

We now assume that, for a given $n\in\nat$, the desired $v_n$ exists. Given $r_1, r_2, r_3\le 3/\gamma$, using \eqref{estimativalinear1},
\begin{align*}
\|Gw_n(t)\|_{3/\gamma}&\lesssim \int_0^t \frac{1}{(t-s)^{\frac{1}{2}\left(\frac{3}{r_1}-\gamma\right)}}\|(v_n\cdot\nabla)v_n\|_{r_1}ds +  \int_0^t \frac{1}{(t-s)^{\frac{1}{2}\left(\frac{3}{r_2}-\gamma\right)}}\|(\phi\cdot\nabla)v_n\|_{r_2}ds \\&+  \int_0^t \frac{1}{(t-s)^{\frac{1}{2}\left(\frac{3}{r_3}-\gamma\right)}}\|(v_n\cdot\nabla)\phi\|_{r_3}ds.
\end{align*}
For the first term, one chooses $r_1=3/(\gamma+1)$ and applies Hölder's inequality:
$$
\|(v_n(t)\cdot\nabla)v_n(t)\|_{r_1}\le \|v_n(t)\|_{3/\gamma}\|\nabla v_n(t)\|_3 \lesssim K_nK_n' \frac{e^{2Lt}}{t^{1-\gamma/2}}.
$$
By Theorem \ref{globexistd2}, there exists $M>0$ such that
\begin{equation}
\sup_{0\le t\le T^*}\|\phi(t)\|_{1,\infty} \le M.
\end{equation}
Then, choosing $r_2=3$ and $r_3=3/\gamma$,
\begin{align*}
\|Gv_n(t)\|_{3/\gamma}&\lesssim K_nK_n'\int_0^t \frac{e^{2Ls}}{(t-s)^{\frac{1}{2}}}\frac{1}{s^{1-\gamma/2}}ds + MK'_n \int_0^t \frac{e^{Ls}}{(t-s)^{(1-\gamma)/2}}\frac{1}{s^{1/2}}ds +  MK_n \int_0^t \frac{e^{Ls}}{s^{(1-\gamma)/2}}ds 
\end{align*}
Set $p>2$. Then
\begin{align*}
\int_0^t \frac{e^{Ls}}{(t-s)^{(1-\gamma)/2}}\frac{1}{s^{1/2}}ds&\le \left(\int_0^t e^{pLs}ds\right)^{\frac{1}{p}}\left(\int_0^t \frac{1}{(t-s)^{(1-\gamma)p'/2}}\frac{1}{s^{p'/2}}ds\right)^{\frac{1}{p'}}\\&\lesssim \frac{e^{Lt}}{(pL)^{\frac{1}{p}}}t^{-(1-\gamma)/2} (T^*)^{\frac{2-p'}{2p'}}.
\end{align*}
Similarly,
\begin{equation}
\int_0^t \frac{e^{Ls}}{s^{(1-\gamma)/2}}ds \lesssim \frac{e^{Lt}}{(pL)^{\frac{1}{p}}}t^{-(1-\gamma)/2} (T^*)^{\frac{1}{p'}}
\end{equation}
Hence
\begin{align*}
\|Gv_n(t)\|_{3/\gamma}\lesssim e^{Lt}t^{-(1-\gamma)/2}\left(e^{LT^*}K_nK_n' + \frac{(T^*)^{\frac{2-p'}{2p'}}}{(pL)^{\frac{1}{p}}}MK_n'+\frac{(T^*)^{\frac{1}{p'}}}{(pL)^{\frac{1}{p}}}MK_n\right)
\end{align*}
This implies that
\begin{equation}\label{desigualdadekn}
K_{n+1}\lesssim K_1 + e^{LT^*}K_nK_n' + \frac{(T^*)^{\frac{2-p'}{2p'}}}{(pL)^{\frac{1}{p}}}MK_n'+\frac{(T^*)^{\frac{1}{p'}}}{(pL)^{\frac{1}{p}}}MK_n.
\end{equation}

We estimate $\nabla Gw_n(t)$ in a similar fashion:
\begin{align*}
\|\nabla Gw_n(t)\|_{3}&\lesssim \int_0^t \frac{1}{(t-s)^{\frac{1}{2}\left(1+\frac{3}{r_1}-\gamma\right)}}\|(w_n\cdot\nabla)w_n\|_{r_1}ds + \lambda \int_0^t \frac{1}{(t-s)^{\frac{1}{2}\left(1+\frac{3}{r_2}-\gamma\right)}}\|(\psi\cdot\nabla)w_n\|_{r_2}ds \\&+ \lambda \int_0^t \frac{1}{(t-s)^{\frac{1}{2}\left(1+\frac{3}{r_3}-\gamma\right)}}\|(w_n\cdot\nabla)\psi\|_{r_3}ds\\&\lesssim K_nK_n'\int_0^t \frac{e^{2Ls}}{(t-s)^{(1+\gamma)/2}}\frac{1}{s^{1-\gamma/2}}ds  +  MK'_n \int_0^t \frac{e^{Ls}}{(t-s)^{1/2}}\frac{1}{s^{1/2}}ds \\&+  MK_n \int_0^t \frac{e^{Ls}}{(t-s)^{\gamma/2}}\frac{1}{s^{(1-\gamma)/2}}ds\\&\lesssim e^{Lt}t^{-1/2}\left(e^{LT^*}K_nK_n' + \frac{(T^*)^{1/p'}}{(pL)^{\frac{1}{p}}}MK_n' +\frac{(T^*)^{1/p'}}{(pL)^{\frac{1}{p}}}MK_n \right).
\end{align*}
Consequently,
\begin{equation}\label{desigualdadekn'}
K'_{n+1}\lesssim K_1 +e^{LT^*}K_nK_n' + \frac{(T^*)^{1/p'}}{(pL)^{\frac{1}{p}}}MK_n' +\frac{(T^*)^{1/p'}}{(pL)^{\frac{1}{p}}}MK_n.
\end{equation}
The recurrence relations \eqref{desigualdadekn} and \eqref{desigualdadekn'} can be solved easily: setting $\mathcal{K}_n=\max\{K_n,K_n'\}$, the worst-case scenario is 
\begin{equation}
\mathcal{K}_{n+1}=f(\mathcal{K}_n)=: C\|v_0\|_3 + C\left[e^{LT^*}\mathcal{K}_n + \frac{(T^*)^{1/p'} + (T^*)^{\frac{2-p'}{2p'}}}{(pL)^{\frac{1}{p}}}M \right]\mathcal{K}_n,
\end{equation}
where $C$ depends only on $p$. For $L$ large enough (depending on $T^*$ and $M$) and for $\|v_0\|_3<\epsilon$ small, 
the above recurrence gives
$$
K_n, K_n'\lesssim \mathcal{K}=2\|v_0\|_3, \ n\in\nat.
$$
\textit{Step 3. Convergence.} Take $w_n=v_n-v_{n-1}$. Setting 
$$
W_n=\sup_{0\le t\le T^*, 0<\gamma\le 1} \left\{e^{-Lt}t^{(1-\gamma)/2}\|w_n(t)\|_{3/\gamma}\right\}
$$
and following the same procedure as in Step 2,
$$
W_{n+1}\le C\left(e^{LT^*}\mathcal{K} + \frac{(T^*)^{1/p'} + (T^*)^{\frac{2-p'}{2p'}}}{(pL)^{\frac{1}{p}}}M\right)W_n\le \frac{1}{2}W_n \le \frac{1}{2^n}W_1
$$
Hence
$$
v_n \to v \mbox{ on } BC([0,T^*], L^{3}(\real^3)),\quad e^{-Lt}t^{(1-\gamma)/2}v_n \to e^{-Lt}t^{(1-\gamma)/2}v \mbox{ on } BC([0,T^*], L^{3/\gamma}(\real^3))
$$
and
$$
e^{-Lt}t^{1/2}\nabla v_n \to e^{-Lt}t^{1/2}\nabla v \mbox{ on } BC([0,T^*], (L^{3}(\real^3))^3).
$$
These convergences then imply that
$$
Gv_n \to Gv \mbox{ on } BC([0,T^*], L^{3}(\real^3)),
$$
and so $v=v_1+Gv$, that is, $v$ is a solution of \eqref{duhamelv}. It is now easy to check the claimed decay of $\nabla v$ using the Duhamel's formula for $v$. For the $L^\infty$ norm, one estimates
$$
\|v(t)\|_\infty^2\lesssim \|v(t)\|_{6}\|\nabla v(t)\|_6 \lesssim \frac{\mathcal{K}^2}{t}.
$$
\textit{Step 4. Uniqueness.} Assume that $v_1, v_2$ verify \eqref{duhamelv}. Then, taking the difference $w=v_1-v_2$, it suffices to check that $w\equiv 0$ close to $t=0$. Given $\delta>0$, there exists $t_0>0$ such that, for $t<t_0$,
$$
t^{(1-\gamma)/2}\|v_1(t)\|_{3/\gamma},t^{(1-\gamma)/2}\|v_2(t)\|_{3/\gamma}\le \delta.
$$
An analogous computation to that of Step 3 gives
\begin{align*}
e^{-Lt}t^{(1-\gamma)/2}\|w(t)\|_{3/\gamma}&\le C\left(e^{Lt_0}\delta + \frac{(t_0)^{1/p'} + (t_0)^{\frac{2-p'}{2p'}}}{(pL)^{\frac{1}{p}}}M\right)e^{-Lt}t^{(1-\gamma)/2}\|w(t)\|_{3/\gamma}\\&\le \frac{1}{2}e^{-Lt}t^{(1-\gamma)/2}\|w(t)\|_{3/\gamma},
\end{align*}
for $t_0, \delta>0$ small enough. This implies $w(t)=0$ for $t<t_0$, which completes the proof.
\end{proof}
\begin{teo}[$L^3$-stability of spatial plane waves]
Given $\phi_0\in X_c^s$, there exists $\epsilon>0$ such that, for any $v_0\in L^3(\real^3)$ with $\|v_0\|_3<\epsilon$, the solution $u$ of (NS) with initial data $v_0+\phi_0$ is global and satisfies
$$
\|u(t)-\phi(t)\|_{L^p(\real^3)}\lesssim \frac{1}{t^{(1-3/p)/2}},\ t>0, \ 3\le p<\infty.
$$
\end{teo}
\begin{proof}
First of all, for any $\delta>0$ fixed, there exists $t_\delta$ such that the $L^2$ norm of the profile of $\phi$ is smaller than $\delta$. It then follows from Lemma 3 that, for $t>t_\delta$,
$$
\|\phi(t)\|_\infty\lesssim \frac{2\delta}{(t-t_\delta)^{1/2}}.
$$
From Theorem \ref{lwpE}, it follows that, for $\|v_0\|_3<\epsilon$ small enough, the solution $u$ of (NS) with initial data $\phi_0+v_0$ exists up to $T=t_\delta$ and $\|v(t_\delta)\|_3<2\epsilon$. We then take $t_\delta$ as our starting point and prove global existence for $v$. For the sake of simplicity, we redefine $v_0:=v(t_\delta)$ and $\phi_0:=\phi(t_\delta)$.

Consider the space
$$
\mathcal{E}=\{v\in C([0,\infty), L^3(\real^3)): t^{\frac{1}{2}-\frac{3}{2p}}\|v(t)\|_p\le M, \forall p\ge3\}.
$$
endowed with the distance
$$
d(v,w)=\sup_{p>3, t>0} \{t^{\frac{1}{2}-\frac{3}{2p}}\|v(t)-w(t)\|_p\} + \|v-w\|_{L^\infty((0,\infty);L^3(\real^3)}.
$$
For any $v\in\mathcal{E}$, set
$$
\Phi(v)=e^{t\Delta}v_0 - \int_0^t e^{(t-s)\Delta}\proj\left((v\cdot\nabla)v+(\phi\cdot\nabla)v+(v\cdot\nabla)\phi\right)(s)ds.
$$
We shall prove that, under suitable choices on $\delta$ and $M$, $\Phi:\mathcal{E}\to\mathcal{E}$ is a strict contraction, yielding the result. To this end, one must write the nonlinear terms as
$$
(v\cdot\nabla)v+(\phi\cdot\nabla)v+(v\cdot\nabla)\phi=\nabla\cdot\left((v+\phi)\otimes (v+\phi) - \phi\otimes\phi\right).
$$
Using this, one is able to estimate $\Phi$ without involving any derivatives of $v$ and $\phi$. Indeed, for any $p\ge 3$, it follows from \eqref{estimativalinear1} that
\begin{align*}
\|\Phi(v(t))\|_p&\le \frac{C\|v_0\|_3}{t^{\frac{1}{2}-\frac{3}{2p}}} + C\int_0^t \frac{1}{(t-s)^{1/2}}\left\|(v+\phi)\otimes(v+\phi)-\phi\otimes\phi\right\|_p(s) ds\\
&\le C\frac{\|v_0\|_3}{t^{\frac{1}{2}-\frac{3}{2p}}} + C\int_0^t \frac{1}{(t-s)^{1/2}}\left(\|v(s)\|_{2p}^2 + \|\phi(s)\|_{\infty}\|v(s)\|_p\right) ds\\&\le C\frac{\|v_0\|_3}{t^{\frac{1}{2}-\frac{3}{2p}}} + C\int_0^t \frac{1}{(t-s)^{1/2}}\frac{M^2+2\delta M}{s^{1-\frac{3}{2p}}}ds\\ &\le \frac{1}{t^{\frac{1}{2}-\frac{3}{2p}}}C\left(\epsilon + M^2+\delta M\right).
\end{align*}
Thus $\Phi(v)\in\mathcal{E}$ if 
$$C(\epsilon+M^2)\le (1-C\delta)M.$$
This condition can be verified if $\epsilon\ll 1$, $\delta<1/2C$ and $M<1/2C$. Moreover,
\begin{align*}
\|\Phi(v(t))-\Phi(w(t))\|_p&\le C\int_0^t \frac{1}{(t-s)^{1/2}}\|(v+\phi)\otimes(v+\phi) - (w+\phi)\otimes(w+\phi)\|_p(s)ds\\&\le  C\int_0^t \frac{1}{(t-s)^{1/2}}\left(\|v(s)\|_{2p} + \|w(s)\|_{2p}\right)\|v(s)-w(s)\|_{2p} \\&+ \int_0^t\frac{1}{(t-s)^{1/2}}\|\phi(s)\|_\infty\|v(s)-w(s)\|_{p} ds\\&\le C\left(\int_0^t \frac{1}{(t-s)^{1/2}}\frac{M+2\delta}{s^{1-\frac{3}{2p}}}ds\right)d(v,w),
\end{align*}
which implies that
$$
d(\Phi(v),\Phi(w)) \le C\left(M+2\delta\right)d(v,w).
$$
Hence, if $C\left(M+2\delta\right)<1/2$, $\Phi$ is a strict contraction over $\mathcal{E}$, which concludes the proof.
\end{proof}
\section{Further comments}
\subsection{Superposition of spatial plane waves}
One could follow the ideas in \cite{ms2} to build a local well-posedness theory for either a numerable or a continuous superposition of spatial plane waves. In the first case, the method should follow the same lines as the single wave case. However, due to the interaction between waves with different speeds, the corresponding stability result should include a smallness condition on the waves themselves (and not simply on the remainder).

In the continuous case, as observed in \cite{ms2}, one solves a linear heat equation for the continuous superposition of plane waves and then solves the remainder equation. In the dispersive case, the fact that the linear solution was a continuous plane wave had an important role in its decay estimates and integrability properties. In the diffusive case, however, the regularization given by the heat kernel already gives the required estimates, even if the linear solution is not a continuous plane wave. The conclusion is that, setting
$$
E=L^3(\real^3) + L^p_z(L^{2p'}_{x,y}(\real^3)), 1<p<\infty,
$$
one should be able to derive directly a local well-posedness result over $E$ and also a global result for small data, using only Kato's method: one simply takes the linear evolution of the component in $L^p_zL^{2p'}_{x,y}$ and then solves the remainder equation in $L^3$.
\subsection{The complex Ginzburg-Landau equation}
Consider the initial value problem for the complex Ginzburg-Landau equation over $\real^d$:
\begin{equation}\label{CGL}\tag{CGL}
u_t=(\epsilon+i)\Delta u + (i-k)|u|^2u, \quad u(0)=u_0,\quad \epsilon, k>0.
\end{equation}
We claim that \eqref{CGL} shares many properties with (NS) (see \cite{ginibrevelo1}, \cite{ginibrevelo2}): in fact,
\begin{itemize}
\item the linear semigroup $\{S(t)\}_{t\ge 0}$ for (CGL) can be seen to be a convolution with a kernel which shares the properties of the heat kernel. Hence the decay estimates of the linear part will be the same and one has $\|S(t)u_0\|_2\to 0$ as $t\to\infty$ (this is easily seen on the Fourier side);
\item for $d=2$, one has global well-posedness for
$$
u_0\in L^2(\real^2), \mbox{ giving } u\in C([0,\infty);L^2(\real^2))\cap C((0,\infty);L^\infty(\real^2))
$$
or
$$
u_0\in H^2(\real^2), \mbox{ giving } u\in C([0,\infty);H^2(\real^2)).
$$
\item for $d=3$ and small viscosity $\epsilon\ll 1$, global well-posedness for generic initial data is unknown. In this case, one might expect that the Schrödinger part of the equation (which is focusing) may give birth to the formation of singularities (see \cite{sverak} for strong numerical evidence). 
\item the nonlinearity $|u|^2u$ has the $L^3(\real^3)$-critical decay behaviour verified by the nonlinearity of (NS):
$$
\||S(t)u_0|^2S(t)u_0\|_{L^3(\real^3)}\lesssim \frac{1}{t}\|u_0\|_{L^3(\real^3)},\quad \|\mathbb{P}(e^{tA}u_0\cdot \nabla)e^{tA}u_0\|_{L^3(\real^3)}\lesssim \frac{1}{t}\|u_0\|_{L^3(\real^3)}.
$$
\end{itemize}
With these similarities in mind, one can easily understand how to obtain the analogous results of existence and stability of spatial plane waves for (CGL) for $d=3$: first of all, considering the spatial plane wave ansatz $u(t,x,y,z)=f(t,w,z)$, $w=(x-cy)/\sqrt{1+c^2}$, one sees that $f$ must satisfy (CGL) in dimension two. Using the global well-posedness results, we obtain globally defined spatial plane waves in $\real^3$.

To prove the local existence on
$$
\tilde{E}= L^3(\real^3)\oplus \tilde{X}^2_c,\ \tilde{X}^2_c:=\left\{ \phi\in L^1_{loc}(\real^3): \phi(x,y,z)=f\left(\frac{x-cy}{\sqrt{1+c^2}}, z\right) \mbox{a.e.}, f\in H^2(\real^2) \right\},
$$
one may proceed as in the (NS) case (since the proof relies only on the decay estimates for the semigroup and on the nature of the nonlinearity). To prove stability, one needs the $L^\infty$ decay of the plane wave. Using Kato's method, it is easy to prove that, for $f_0\in L^2(\real^2)$ small enough, there exists a unique solution $f$ of (CGL) with initial data $f_0$ such that
$$
t^{1/2-1/p)}f\in BC([0,\infty); L^p(\real^2)),\quad 2\le p\le \infty,
$$
which settles the $L^\infty$ decay of the wave for small $L^2$ profiles. All that remains is to check that, for any $f_0\in H^2(\real^2)$, the corresponding $H^2$ solution satisfies $\|f(t)\|_2\to 0$. Indeed, one has
$$
\frac{1}{2}\frac{d}{dt}\|f(t)\|_2^2 = -\epsilon \|\nabla f(t)\|_2 - k\|f(t)\|_4^4,\mbox{ thus impliying } \lim_{\tau\to \infty}\int_{\tau}^\infty  \|\nabla f(t)\|_2 +\|f(t)\|_4^4dt=0.
$$
and so, given $\delta>0$, there exist $t, \tau>0$ such that
\begin{align*}
\|f(t)\|_2^2&\lesssim \|S(t)f(\tau)\|_2^2 + \int_{\tau}^t \|f(s)\|_6^3ds \lesssim \|S(t)f(\tau)\|_2^2 + \int_{\tau}^t \|\nabla f(s)\|_2 +\|f(s)\|_4^4ds\\&\lesssim \|S(t)f(\tau)\|_2^2 +\delta \lesssim 2\delta.
\end{align*}
In conclusion, we obtain the $L^3$-stability of spatial plane waves for the complex Ginzburg-Landau equation:
\begin{prop}
Given $\phi_0\in \tilde{X}_c^2$, there exists $\epsilon>0$ such that, for any $v_0\in L^3(\real^3)$ with $\|v_0\|_3<\epsilon$, the solution $u$ of (CGL) with initial data $v_0+\phi_0$ is global and satisfies
$$
\|u(t)-\phi(t)\|_{L^p(\real^3)}\lesssim \frac{1}{t^{(1-3/p)/2}},\ t>0, \ 3\le p<\infty.
$$
\end{prop}
\section{Acknowledgements}
Mário Figueira was partially supported by Fundação para a Ciência e Tecnologia, through the grant UID/MAT/04561/2013. Simão Correia was also supported by Fundação para a Ciência e Tecnologia, through the grants SFRH/BD/96399/2013 
and UID/MAT/04561/2013.

\small
\noindent \textsc{Sim\~ao Correia}\\
CMAF-CIO and FCUL \\
\noindent Campo Grande, Edif\'icio C6, Piso 2, 1749-016 Lisboa (Portugal)\\
\verb"sfcorreia@fc.ul.pt"\\

\small
\noindent \textsc{Mário Figueira}\\
CMAF-CIO and FCUL \\
\noindent Campo Grande, Edif\'icio C6, Piso 2, 1749-016 Lisboa (Portugal)\\
\verb"msfigueira@fc.ul.pt"\\


\begin{thebibliography}{99}

\bibitem{ms1} S. Correia, M. Figueira, \emph{The hyperbolic nonlinear Schrödinger equation}, arXiv:1510.08745

\bibitem{ms2} S. Correia, M. Figueira, \emph{Spatial plane waves for the nonlinear Schrödinger equation: local existence and stability results}, arXiv:1603.00771

\bibitem{ginibrevelo1} J. Ginibre, G. Velo, \emph{The Cauhcy problem in local spaces for the complex Ginzburg-Landau equation. I. Compactness methods}, Phys. D 95 (3-4), 191-228 (1996)

\bibitem{ginibrevelo2} J. Ginibre, G. Velo, \emph{The Cauhcy problem in local spaces for the complex Ginzburg-Landau equation. II. Contraction methods}, Comm. Math. Phys. 187, 45-79 (1997)

\bibitem{sverak} P. Plechá\v c, V. \v Sverák, \emph{On self-similar solutions of the complex Ginzburg-Landau equation}, Comm. Pure Appl. Math. 54, 1215-1242 (2001)

\bibitem{lemarie} P. Lemarié-Rieusset, \emph{Recent developments in the Navier-Stokes problem}, Chapman\& Hall/CRC (2002)

\bibitem{kato} T. Kato, \emph{Strong $L^p$- Solutions of the Navier-Stokes Equation in $\real^m$, with Applications to Weak Solutions}, Math. Zeit. 187, 471-480 (1984)

\bibitem{katofujita} T. Kato, H. Fujita, \emph{On the non-stationary Navier-Stokes system}, Rend. Sem. Mat. Univ. Padova 32, 243-260 (1962)



\end{thebibliography}
\end{document}